\documentclass[10pt]{article}
\usepackage[top=1in, bottom=1in, left=1.25in, right=1.25in]{geometry}
\usepackage{amsmath,amssymb,amsthm,xcolor,enumerate,bbm,array}
\usepackage[unicode,breaklinks=true,colorlinks=true]{hyperref}

\usepackage{theoremref}
\usepackage{mathrsfs}
\usepackage[]{authblk}
\usepackage{comment}
\usepackage{cancel}
\usepackage{orcidlink}
\usepackage{graphicx}
\usepackage{float}
\usepackage{changepage}
\usepackage{array,multirow,makecell} 
\numberwithin{equation}{section}
\newtheorem{thm}{Theorem}[section]
\newtheorem{cor}[thm]{Corollary}
\newtheorem{lem}[thm]{Lemma}
\newtheorem{prop}[thm]{Proposition}

\newtheorem{remark}[thm]{Remark}
\theoremstyle{definition}

\renewcommand{\div}{\mathop{\rm div}\nolimits}

\newcommand{\esssup} {\mathop{\mathrm{ess\,sup}}}

\allowdisplaybreaks
\begin{document}
	\title{Global well-posedness of strong solutions to the initial-boundary value problem for  a two-dimensional
stress-diffusive Oldroyd-B model  in the creeping flow regime} 
		\author[1]{\rm Yinghui Wang\orcidlink{0000-0002-7565-5525}}
		\author[2]{\rm Shihao Zhang}
		\author[3]{\rm Zhuo Zhang\thanks{Corresponding author.\\ $\quad\quad\quad$ E-mail addresses:  yhwangmath@163.com (Y.H. Wang), zhangshihao@hunnu.edu.cn (S.H. Zhang),  zhangzhuo0614@Outlook.com (Z. Zhang).}}
		\affil[1,2,3]{\footnotesize  MOE-LCSM, School of Mathematics and Statistics, Hunan Normal University, Changsha, Hunan 410081, P. R. China} 
		%
		\date{}
		\maketitle

		
		\begin{abstract}
This paper investigates the global well-posedness of strong solutions to a stress-diffusive
Oldroyd-B system  in two-dimensional smooth bounded domains in the zero Reynolds number (creeping flow) regime. The stress-diffusion term is kept explicit throughout the paper and is understood as the usual center-of-mass diffusion regularization of the Oldroyd-B constitutive equation. In the corresponding non-diffusive creeping-flow setting, the strongest available result is a Beale-Kato-Majda type breakdown criterion for the three-dimensional Cauchy problem due to Kupferman, Mangoubi and Titi [Commun. Math. Sci. 6 (2008)], and global well-posedness remains open even in two dimensions. We prove the global existence and uniqueness of strong solutions for arbitrarily large $H^1$ initial polymeric stresses satisfying the natural non-negativity condition on the conformation tensor. The result covers the initial-boundary value problem on general smooth bounded domains and shows how the Stokes elliptic structure, the preservation of the non-negativity of the conformation tensor, and the stress diffusion combine to close large-data estimates at the $H^1$ level. We also point out a regularity feature specific to the zero Reynolds number regime: the velocity field gains higher spatial regularity from the elliptic Stokes equation than the polymeric stress tensor.
  
			\vspace{4mm}
			
			
		\end{abstract}
		
		\section{Introduction}
		The Oldroyd-B model serves as a fundamental framework for describing the dynamics of viscoelastic fluids and has found extensive applications in fields such as petroleum extraction, polymer processing, and biomedical engineering. Further applications of this model can be found in the references \cite{Bird_1}. The original Oldroyd-B model was introduced by Oldroyd (\cite{Oldroyd_1958}) to describe mathematically viscoelastic effects of certain types of fluids. This type of
fluids is described by the following equations.
  \begin{equation}\label{Oldroyd_B}
			\left\{\begin{aligned} 
				&\mathrm{Re}\big(\partial_t u+(u\cdot\nabla) u\big)+\nabla p-(1-\alpha)\Delta u= \mathrm{div}\tau, \\
				 &\mathrm{We}\big(\partial_t\tau+(u\cdot\nabla)\tau+g_a(\tau,\nabla u) \big)-\varepsilon\Delta\tau+ \tau=2\alpha \mathbb{D}(u), \\
				&\mathrm{div} u=0.
			\end{aligned}\right.
		\end{equation}
		Here, $u$ denotes the fluid velocity, $p$ is the pressure
function of the fluid,  $\tau$ is the elastic part of the stress tensor. The dimensionless numbers $\mathrm{Re}>0$ and $\mathrm{We}>0$ are the Reynolds number and the Weissenberg number, respectively. The constant $\alpha\in (0,1)$ is the retardation parameter, $\varepsilon>0$ is the center-of-mass diffusion coefficient. The quadratic term $g_a(\tau,\nabla u)$ is given by \begin{equation*}
			g_a(\tau,\nabla u)=\tau \mathbb{W}(u) -\mathbb{W}(u)\tau - a(\mathbb{D}(u)\tau+\tau \mathbb{D}(u)),~~\text{for some}~~a \in [-1,1],
		\end{equation*} 
		where $\mathbb{D}(u) = \frac{1}{2}(\nabla u + \nabla^\top u)$ and $\mathbb{W}(u) = \frac{1}{2}(\nabla u - \nabla^\top u)$ denote the deformation tensor and the vorticity tensor, respectively. Here, the velocity gradient is defined by the components  $(\nabla{u})_{ij}=\partial_{j}u_i$.\footnote{Note that in some references, such as \cite{Kupferman_2007}, the author uses the opposite notation, namely $(\nabla{u})_{ij}=\partial_{i}u_j$.}

  In the creeping flow regime (i.e., when the Reynolds number $\mathrm{Re}$ is sufficiently small), the system \eqref{Oldroyd_B} can be approximated by the following system:\footnote{For simplicity we take the stretching parameter $a=1$ and absorb the Weissenberg number into the relaxation rate $\gamma>0$.} 
  \begin{equation}\label{EQ-all} 
			\left\{\begin{aligned} 
				&- (1 - \alpha) \Delta u + \nabla p = \operatorname{div} \tau  , \\
				&\partial_{t} \tau + (u \cdot \nabla) \tau - (\nabla u\tau+ \tau\nabla^\top u  ) + \gamma\tau -\varepsilon\Delta\tau= 2\alpha \mathbb{D}(u),\\
                &\div u = 0.
			\end{aligned} \right.
		\end{equation} 
        The zero-Reynolds-number reduction is standard in slow viscoelastic flows, where the velocity is determined instantaneously by the stress through a Stokes system; see, for instance, \cite{Kupferman_2007}. The term $-\varepsilon\Delta\tau$ is the center-of-mass, or stress-diffusion, regularization used in several physical and numerical models of polymeric fluids and can describe shear-banding and vorticity-banding effects; see \cite{Bhave2,Dhont_2008,Malek_etal_2018}. Thus the present paper concerns this mathematically and physically motivated stress-diffusive creeping-flow Oldroyd-B system.
         
 Let $\Omega\subset\mathbb{R}^2$ be a smooth bounded domain.
         Following the thermodynamically consistent stress-diffusion framework discussed in \cite{Malek_etal_2018}, we supplement system \eqref{EQ-all} with the following initial and boundary conditions. The initial condition is given by
		\begin{align}\label{initial-con}
			 \tau(x,y,0) = \tau_{0}(x,y),
		\end{align}		
		for $(x,y)\in \Omega$; and  the boundary conditions are  
		\begin{align}\label{boundary-con}
			u = 0,\quad \partial_n\tau = 0,
		\end{align} 
  for $(x,y)\in \partial \Omega$ and $t\in\mathbb{R}_+$.  

   To the best of the authors' knowledge, the first breakdown criterion for the
non-diffusive Oldroyd--B system in the creeping-flow regime was established by
Kupferman, Mangoubi, and Titi \cite{Kupferman_2007}.  For the three-dimensional
Cauchy problem, they proved local well-posedness in $H^m$, $m>5/2$, and showed
that a finite maximal existence time $T^*$ necessarily satisfies
\begin{equation*}
 \lim_{t\uparrow T^*}\int_0^t\|\tau(s)\|_{L^\infty}\,{\rm d}s=\infty.
\end{equation*}
Moreover, as shown in \cite{Elgindi_Masmoudi_2020}, the two-dimensional zero diffusion Oldroyd-B model is   mildly $L^\infty$ 
ill-posed (see Theorem 9.1 in \cite{Elgindi_Masmoudi_2020}). These observations motivate the present focus on the stress-diffusive system with fixed $\varepsilon>0$: under the natural assumption that the initial conformation tensor is nonnegative definite, one can ask whether the regularized creeping-flow model admits a large-data strong solution theory. This is the problem addressed here.\\[0.5mm]
		
Next, we review the research progress concerning Oldroyd-B type models. For the {\bfseries non-diffusive} case ($\varepsilon=0$ in \eqref{Oldroyd_B}), there is an extensive body of literature regarding the global well-posedness of small strong solutions, blow-up criteria, and the long-time behavior of solutions in various domains of $\mathbb{R}^2$ and $\mathbb{R}^3$, see, for instance, \cite{Chemin_2001,Fang_Hieber_Zi_2013,Fernandez_Guillen_Ortega_1998,Guillop_1990,Hieber_Wen_Zi_2019,Huang_2022,Lei_Zhou_2005,Lei_Masmoudi_Zhou_2010,Sun_Zhang_2011,Zi_Fang_Zhang_2014} and the references therein. Notably, the only existing result concerning global-in-time large solutions was established by Lions and Masmoudi \cite{Lions_Masmoudi_2000}, who proved the global existence of two- and three-dimensional weak solutions for the corotational Oldroyd-B model, although the uniqueness of such solutions remains an open problem. For a more comprehensive survey, we refer the reader to the lectures by Saut \cite{Saut_2013}.

Regarding the {\bfseries diffusive} Oldroyd-B model ($\varepsilon>0$ in \eqref{Oldroyd_B}), research has advanced significantly since the pioneering work of Barrett and Boyaval \cite{Barrett_Boyaval_2011} on the global existence of weak solutions. In \cite{Constantin_2012}, Constantin and Kliegl established the global existence and uniqueness of strong solutions for the two-dimensional Cauchy problem, provided that the initial conformation tensor $\tau_0+\alpha\mathbb{I}$ is non-negative. This result was later generalized by Nong et al. \cite{Nong_2025} to initial-boundary value problems in 2D strip domains. For the inviscid case ($\alpha = 1$),  Elgindi and Rousset \cite{Elgindi_Rousset_2015} obtained the global well-posedness of strong  solution for the 2D Cauchy problem under the assumption of sufficiently small initial data, a result subsequently extended to the 3D case by Elgindi and Liu \cite{Elgindi_Liu_2015}. More recently, Wang and Wen \cite{Wang_Wen_2024} investigated the vanishing diffusion limits for 2D initial-boundary value problems featuring boundary layers.
 
 Compared with these works, the present paper focuses on the zero Reynolds number system and on the initial-boundary value problem in a general bounded smooth domain. The absence of the velocity time derivative changes the structure of the estimates: the velocity is recovered from the stress through a Stokes problem at each time, while the polymeric stress still contains the nonlinear stretching terms $\nabla u\tau+\tau\nabla^\top u$. The main issues are therefore to prove that the conformation tensor remains nonnegative definite under Neumann stress boundary conditions, to exploit the Stokes elliptic gain without imposing extra compatibility conditions on the initial velocity, and to close large-data estimates starting only from $H^1$ regularity of the initial stress.

 The main contributions of this paper can be summarized as follows.
 \begin{enumerate}[(i)]
    \item We establish local strong well-posedness for the elliptic-parabolic initial-boundary value problem \eqref{EQ-all}-\eqref{boundary-con} with merely $H^1$ initial stress. 
    \item  Under the natural assumption that the initial conformation tensor $\tau_0+\alpha\mathbb{I}$ is positive semidefinite, we prove that the conformation tensor remains nonnegative definite and use this property to derive a trace-based $L^1$ estimate. 
    \item  Combining the $L^1$ control, the stress diffusion, and the Stokes elliptic regularity, we obtain global-in-time strong solutions of \eqref{EQ-all}-\eqref{boundary-con}  for arbitrarily large $H^1$ initial data in two-dimensional bounded smooth domains. 
    \item  We identify a regularity hierarchy specific to the creeping-flow regime: the velocity field is smoother than the polymeric stress tensor because it is determined by an elliptic Stokes problem. 
\end{enumerate}
\subsection{Notations}
		In this section, we  introduce some notations that will be repeatedly used throughout the paper.
		\begin{enumerate}[(1)] 
			\item  $C$ denotes a generic positive constant which may depend on $\varepsilon,\alpha,\gamma$ and $\Omega$, but is independent of temporal variable  $t$. And, $A\lesssim B\Leftrightarrow A\leq CB$.
			\item For vector function $u$ and matrix-valued function $\tau$, we use the following notations:
			$$ (\nabla u)_{ij}=\partial_{j}u_{i},\; (\nabla^\top u)_{ij}=\partial_{i}u_{j},\;  (\nabla\tau)_{nml}=\partial_{l}\tau_{nm}.$$
           \item  The summation convention is adopted: whenever an index appears twice in a single term,  summation over that index is implied. 
             \item $L^p$ and $H^s$ denote the usual Lebesgue and Sobolev spaces over $\Omega\subset\mathbb{R}^2$. Let $X$ be a Banach space. For given $T>0$, we use the notation that
			 \begin{equation*}
				\|u\|_{L_T^p X}:= 
				\begin{cases} \displaystyle\left(\int_{0}^{T}  \|u(t)\|_X^p  {\rm d}t\right)^{\frac{1}{p}},& 1\leq p<\infty.\\
					\displaystyle\esssup_{t\in[0,T]}\|u(t)\|_X  ,& p=\infty.
				\end{cases}
			\end{equation*}  
		\end{enumerate}
\subsection{Main results} 
The first result of the present paper is the local well-posedness of strong solutions for problem \eqref{EQ-all}-\eqref{boundary-con}.
\begin{prop}\label{Local_Existence}
Let $\Omega\subset\mathbb{R}^2$ be a bounded connected domain of class
$C^{2,1}$, let $0<\alpha<1$ and $\varepsilon,\gamma>0$, and assume that
$\tau_0\in H^1(\Omega)$.  Then there exists a time
$T_*>0$, depending only on $\Omega$, the parameters, and
$\|\tau_0\|_{H^1}$, such that \eqref{EQ-all}--\eqref{boundary-con} has a
unique strong solution $(u,\tau)$ on $[0,T_*]$ satisfying
\begin{align*}
&u\in C([0,T_*];H^2(\Omega))\cap L^2(0,T_*;H^3(\Omega)),
 \qquad u_t\in L^2(0,T_*;H^1(\Omega)),\\ 
&\tau\in C([0,T_*];H^1(\Omega))\cap L^2(0,T_*;H^2(\Omega)),
 \qquad \tau_t\in L^2(0,T_*;L^2(\Omega)).
\end{align*}
\end{prop}
         \begin{remark}
         In \cite{Kupferman_2007}, the authors used the Stokes kernel to reduce the non-diffusive Cauchy problem to a closed equation for $\tau$ and then applied the classical arguments of Majda and Bertozzi (\cite{Majda_Bertozzi_2002}) to prove local well-posedness with initial data in $H^{m}(\mathbb{R}^3)$, $m>5/2$. Following their approach, one can obtain local well-posedness for the two-dimensional Cauchy problem with data in $H^{m}(\mathbb{R}^2)$, $m>2$. In the present work, the boundary value problem \eqref{EQ-all}-\eqref{boundary-con} is treated directly as an elliptic-parabolic coupled system. This allows us to obtain local well-posedness from $H^1$ initial stress data. 
		\end{remark}
Based on Proposition \ref{Local_Existence}, and the uniform estimates derived in Section \ref{Sec_est}, we obtain the following global-in-time well-posedness result for problem \eqref{EQ-all}-\eqref{boundary-con}. 
		\begin{thm}\label{main_thm}
			Let $\Omega \subset \mathbb{R}^2$ be a bounded, connected domain with $C^{2,1}$ boundary. Suppose that the initial data satisfy $\tau_{0} \in H^{1}(\Omega)$ and that the tensor $\tau_{0} + \alpha\mathbb{I}$ is  symmetric and nonnegative definite. Then, for any $T>0$, the system \eqref{EQ-all}-\eqref{boundary-con}  possesses a unique strong solution  
$(u,\tau)$ with the following regularity:
\begin{align*}
				&u\in C([0,T];H^2)\cap L^2(0,T;H^3),
                \qquad u_t\in L^2(0,T;H^1),\\ 
				&\tau\in C([0,T];H^1)\cap L^2(0,T;H^2),
                \qquad\tau_t\in L^2(0,T;L^2).
			\end{align*}
		\end{thm}   
        \begin{remark}
            Due to the absence of inertia terms in the velocity equation,  $H^1$ regularity of the initial data is sufficient to guarantee global well-posedness for \eqref{EQ-all}-\eqref{boundary-con}. This contrasts with the results for the full Oldroyd-B model \eqref{Oldroyd_B} (see, e.g., \cite{Constantin_2012,Nong_2025} and references therein). Moreover, for the same reason, the velocity field enjoys higher regularity than the polymeric stress tensor in problem \eqref{EQ-all}-\eqref{boundary-con}, which is again different from the situation for system \eqref{Oldroyd_B}.
        \end{remark} 
         \begin{remark}
             Compared with \cite{Constantin_2012}, which treats the two-dimensional Cauchy problem for a diffusive Oldroyd-B system, and with \cite{Nong_2025}, which treats an initial-boundary value problem in a strip domain, Theorem \ref{main_thm} addresses a zero-Reynolds-number elliptic-parabolic system on general two-dimensional bounded smooth domains. The proof relies on a trace estimate for the conformation tensor, the preservation of its non-negativity, and Stokes regularity in bounded domains. 
         \end{remark}
		The rest of the paper is organized as follows.
		\begin{itemize}
\item In Section \ref{Sec_Prel}, we collect some preliminary results.
\item Section \ref{Sec_Local} is devoted to proving the existence and uniqueness of local strong solutions to system \eqref{EQ-all}-\eqref{boundary-con}.
\item In Section \ref{Sec_est}, we establish uniform {\it a priori} estimates for the solution to system \eqref{EQ-all}-\eqref{boundary-con} and complete the proof of Theorem \ref{main_thm} via the continuity argument.
\end{itemize}
		
\section{Preliminaries} \label{Sec_Prel}
We start this section by recalling the  regularity estimate for   the following Stokes problem
\begin{equation}\label{EQ-t22h}
\begin{cases}
-\Delta u+\nabla p = f, &\text{ in }\Omega,\\
\div u =0, &\text{ in }\Omega,\\
u=0, &\text{ on }\partial\Omega.
\end{cases}
\end{equation} 
\begin{lem}[\cite{Boyer_Fabrie_2013}]\label{stoke}
Let $\Omega$ be a bounded connected domain of class $C^{k+1,1}$ in $\mathbb{R}^2$, with $k\geq 0$. Assume that $f \in H^{k}(\Omega)$. Then, there exists a  
  unique solution $(u, p)\in  H^{k+2}(\Omega) \times H^{k+1}(\Omega)$  of  problem \eqref{EQ-t22h}, after imposing
  $\int_\Omega p\,{\rm d}x=0$,  satisfying  
\begin{equation*} 
\|u\|_{H^{k+2}} + \|p\|_{H^{k+1}} \leq C \|f\|_{H^k}.
\end{equation*} 
Moreover, if $f=\mathrm{div}F$ with $F\in H^1$, then the solution $u$ of \eqref{EQ-t22h} satisfies 
\begin{equation*} 
\|u\|_{H^{1}}  \leq C \|F\|_{L^2}.
\end{equation*} 
\end{lem}
Next, we consider the following Neumann problem of the linear transport diffusion equations.
\begin{equation}\label{EQ-t2h}
\begin{cases}
\partial_t \tau + v \cdot \nabla \tau- \varepsilon \Delta \tau + \gamma \tau =  g, &\text{in } \Omega\times(0,T),\\
\partial_n \tau =  0,  &\text{on } \partial \Omega\times(0,T),\\
\tau(x,y,0) = \tau_0, & \text{on }  \Omega,
\end{cases}
\end{equation} 
where $T>0$ is a given constant.
For problem \eqref{EQ-t2h}, we have the following results.
\begin{lem}\label{lem_tau_para}
Let $\Omega \subset \mathbb{R}^2$ be a bounded and connected domain of class $C^{2,1}$. Suppose that $\tau_0 \in H^1(\Omega)$, $v \in L^{\infty}([0,T]; H^1_0(\Omega))$ with $\mathrm{div}v=0$, and $g \in L^{2}([0,T]; L^{2}(\Omega))$. Then \eqref{EQ-t2h} has a unique solution
\[
 \tau\in C([0,T];H^1)\cap L^2(0,T;H^2),
 \qquad \tau_t\in L^2(0,T;L^2).
\]
Moreover,
\begin{align}\label{tau}
&\|\tau\|_{L_T^\infty H^1}^2+
  \|\tau\|_{L_T^2H^2}^2+
  \|\tau_t\|_{L_T^2L^2}^2 \nonumber\\
&\quad\le C\bigl(\|\tau_0\|_{H^1}^2+
                 \|g\|_{L_T^2L^2}^2\bigr)
 \exp\!\left\{C\left(T+\|v\|_{L_T^2L^4}^2+
                         \|v\|_{L_T^4L^4}^4\right)\right\},
\end{align}
where $C$ depends only on $\Omega$, $\varepsilon$, and $\gamma$.
\end{lem}

\begin{proof} 	
 The existence and uniqueness of the solutions to problem \eqref{EQ-t2h} can be established via the standard Galerkin approximation method (refer to Chapter 7 of \cite{Evans_2010} for details). Here, we only prove regularity estimates \eqref{tau}. Firstly, multiplying $\eqref{EQ-t2h}_1$ by $\tau$ and integrating the result over $\Omega$ by parts, we have
\begin{align}\label{ener-eq-tau}
 \frac{1}{2}\frac{{\rm d}}{{\rm dt}}\|\tau\|_{L^2}^2+\varepsilon\|\nabla\tau\|_{L^2}^2+\gamma\|\tau\|_{L^2}^2
 &=\int g :\tau {\rm d}x{\rm d}y \notag \\ 
 &\leq \|g\|_{L^2}\|\tau\|_{L^2}\leq\frac{1}{2\gamma}\|g\|_{L^2}^2
				+\frac{\gamma}{2}\|\tau\|_{L^2}^2,
 \end{align}
which implies that
\begin{align}\label{EQ-t4h}
&\frac{1}{2}\frac{{\rm d}}{{\rm dt}}\|\tau\|_{L^2}^2+\varepsilon\|\nabla\tau\|_{L^2}^2			+\frac{\gamma}{2}\|\tau\|_{L^2}^2\leq\frac{1}{2\gamma}\|g\|_{L^2}^2.
\end{align}
Integrating \eqref{EQ-t4h} over $[0,t]$ with $0 < t < T$, we obtain
\begin{align}\label{2.6}			&\|\tau(t)\|_{L^2}^2+\gamma\|\tau\|_{L_t^2L^2}^2+2\varepsilon\| \nabla\tau\|_{L_t^2L^2}^2\leq
\|\tau_0\|_{L^2}^2+\frac{1}{\gamma}\left\|g \right\|_{L_t^2L^2}^2.
\end{align}		
Next, multiplying  $\eqref{EQ-t2h}_1$ by $\partial_t\tau-\Delta\tau$ and integrating the results over $\Omega$ by parts, we have
\begin{align}\label{nab-tau-L2-est}
&\frac{1}{2}\frac{{\rm d}}{{\rm dt}} \left[ (\varepsilon + 1)\|\nabla \tau\|^2_{L^2} + \gamma \|\tau\|^2_{L^2} \right] + \varepsilon  \|\nabla^2 \tau\|^2_{L^2} +\|\partial_t\tau\|^2_{L^2} + \gamma\|\nabla\tau\|^2_{L^2}\nonumber\\
&= \int g:(\partial_t \tau-\Delta\tau)- (v\cdot\nabla)\tau:\partial_t \tau+(v\cdot\nabla)\tau: \Delta\tau {\rm d}x{\rm d}y=E_1+E_2+E_3.
\end{align}
Together with H\"{o}lder's  inequality, Young's  inequality and Ladyzhenskaya inequality, we can estimate $E_1$-$E_3$ as follows:
\begin{align}
 |E_1| &\leq \frac{1}{4} \|\partial_t\tau\|_{L^2}^2+\frac{\varepsilon}{6} \|\nabla^2\tau\|_{L^2}^2+C_\varepsilon\|g\|_{L^2}^2.\label{EQ-t6h}\\
|E_2|&\lesssim\|\partial_t \tau\|_{L^2} \|v\|_{L^4} (\|\nabla \tau\|_{L^2}^{\frac{1}{2}} \|\nabla^2 \tau\|_{L^2}^{\frac{1}{2}} + \|\nabla \tau\|_{L^2} )\nonumber\\
&\leq \frac{1}{4}\|\partial_t \tau\|_{L^2}^2 + \frac{\varepsilon}{6}\|\nabla^2\tau\|_{L^2}^2+C_\varepsilon(\|v\|_{L^4}^2+\|v\|_{L^4}^4) \|\nabla \tau\|_{L^2}^2.\label{EQ-t7h}\\
|E_3|&\lesssim \|v\|_{L^4}\|\Delta\tau\|_{L^2} (\|\nabla \tau\|_{L^2}^{\frac{1}{2}} \|\nabla^2 \tau\|_{L^2}^{\frac{1}{2}} + \|\nabla \tau\|_{L^2} )\nonumber\\ 
&\leq \frac{\varepsilon}{6}\|\nabla^2\tau\|_{L^2}^2+C_\varepsilon(\|v\|_{L^4}^2+\|v\|_{L^4}^4) \|\nabla \tau\|_{L^2}^2.
\end{align}
Substituting the estimates for $E_1$-$E_3$ into \eqref{nab-tau-L2-est}, we obtain
\begin{align}
&\frac{\mathrm{d}}{\mathrm{d}t}\left [(\varepsilon + 1)\|\nabla \tau\|_{ L^{ 2 }}^{ 2 }+\gamma\|\tau\|_{L^{ 2 }}^{2} \right]+\varepsilon \|\nabla^2\tau\|_{L^{2}}^{2}+\|\partial_{ t }\tau\|_{L^{2}}^{2} +\gamma\|\nabla\tau\|_{L^{2}}^{2}\nonumber \\ 
&\leq C_\varepsilon\|g\|_{L^2}^2 +C_\varepsilon(\|v\|_{L^4}^2+\|v\|_{L^4}^4) \|\nabla \tau\|_{L^2}^2.\label{4th}
\end{align} 
Applying Gronwall inequality to the above inequality,  we deduce that
\begin{align}\label{EQ-thtau}
&\| \nabla \tau(t) \|_{L^2}^2 +\|\nabla^2\tau\|_{L^2_tL^{2}}^{2}+\|\partial_{ t }\tau\|_{L^2_tL^{2}}^{2} \nonumber\\  
&\leq C_\varepsilon \left(\|\tau_0 \|_{H^1}^2 +  \| g \|_{L^2_tL^2}^2  \right)\exp \left(  C_\varepsilon\int_0^t (\|v\|_{L^4}^2+\|v\|_{L^4}^4) {\rm d}s \right).
\end{align}
Summing \eqref{2.6} and \eqref{EQ-thtau} up, one can complete the proof of \eqref{tau}. 
\end{proof} 
Inspired by the works \cite{Constantin_2012} and \cite{Nong_2025}, to establish the global-in-time well-posedness of problem \eqref{EQ-all}, we need the nonnegativity of $\tau + \alpha \mathbb{I}$. 
Define $\mathbb{T} = \tau + \alpha \mathbb{I}$, after some directly calculation, one can find that problem \eqref{EQ-all}-\eqref{boundary-con} is equivalent to the following problem 
\begin{equation}\label{EQ-t19h}
\begin{cases}
-(1-\alpha) \Delta u + \nabla p = \operatorname{div} \mathbb{T},&  \text{in } \Omega \times (0, T), \\
\mathbb{T}_t + (u \cdot \nabla) \mathbb{T}+ \gamma \mathbb{T} - \varepsilon \Delta \mathbb{T} -(\nabla u\mathbb{T}+\mathbb{T}\nabla^\top u) = \alpha\gamma \mathbb{I}, & \text{in } \Omega \times (0, T), \\
\div u=0, & \text{in } \Omega \times (0, T),\\
u=0, ~~\partial_n \mathbb{T} =0, &\text{on } \partial \Omega\times (0, T),\\
\mathbb{T}(x,y,0) =\mathbb{T}_0 = \tau_0 + \alpha \mathbb{I}, &\text{on } \Omega.
\end{cases}
\end{equation}
No independent initial condition is imposed on $u$, since the velocity is determined at each time by the Stokes equation. 
Then, we have the following Lemma. 
\begin{lem}\label{lem_positivity}
Assume that $\mathbb{T}_0 \in H^1(\Omega)$ is symmetric and nonnegative definite a.e. in $\Omega$, and let $(u,\mathbb{T})$ be a strong solution of problem \eqref{EQ-t19h} with $u\in L^{1}(0,T;W^{1,\infty}(\Omega))$. Then $\mathbb{T}(x,t)$ remains symmetric and nonnegative definite for a.e. $(x,t)\in\Omega\times(0,T)$.
\end{lem}
Based on Lemma \ref{lem_positivity}, we have the following Corollary.
\begin{cor}\label{cor_L1}
Under the assumptions of Theorem \ref{main_thm}, let $(u, \tau)$ be a strong solution of \eqref{EQ-all}-\eqref{boundary-con}. Then $\mathbb{T}=\tau+\alpha\mathbb{I}$ remains nonnegative definite. Furthermore,
$$\|\tau\|_{L^1}\leq\|\mathrm{tr}\mathbb{T}\|_{L^1}+C(\alpha,\Omega).$$
\end{cor}

\begin{remark}
In Lemma 2.2 of \cite{Nong_2025}, the authors proved the nonnegativity of $\mathrm{det}(\mathbb{T})$ for a more complicated system related to \eqref{EQ-t19h}. With slight modifications, the proofs of Lemma \ref{lem_positivity} and Corollary \ref{cor_L1} can be completed in a similar manner. We omit the details here for brevity.
\end{remark}
We also need the following Aubin–Lions-Simon compactness lemma.
\begin{lem}[\cite{Simon_1986}]\label{Simom} 
    Let $X\subset E\subset Y$ be Banach spaces, with the embedding $X\to E$ compact. Then the following embedding is compact:
    \begin{gather*}
        L^\infty(0,T;X)\cap\left\{\phi:\frac{\partial\phi}{\partial t}\in L^r(0,T;Y)\right\}\to C([0,T]; E),~~if~~1<r<\infty.
    \end{gather*}
\end{lem}
\section{Local-in-time well-posedness}\label{Sec_Local}
 This section is devoted to proving the local well-posedness of system \eqref{EQ-all}-\eqref{boundary-con}.   
The proof of Proposition \ref{Local_Existence} can be completed via a standard fixed-point argument. Here, we provide a sketch of the proof for the reader's convenience. To begin with, for $T>0$, we introduce the Banach space 
$X_T=C([0,T];H^{1})\times C([0,T];L^2)$, and define
\begin{align*}
K_{T}=\Big\{& ({v}, {\sigma}) \Big|{\sigma}\in C([0,T];H^{1}(\Omega))\cap L^{2}([0,T];H^{2}(\Omega)),~\partial_t \sigma\in L^{2}([0,T];L^{2}(\Omega)),\\
&~{\sigma}(0)=\tau_{0}\mathrm{~in~} \Omega,~\partial_n\sigma|_{\partial\Omega}=0,~{v}\in C([0,T];H^{2}(\Omega))\cap L^{2}([0,T];H^{3}(\Omega)),~\mathrm{div}v=0,  \\		
&~v|_{\partial\Omega}=0,~\| \sigma \|_{L_T^{\infty} H^1}^2 + \| \sigma\|_{L_T^2 H^2}^2 + \| \partial_t \sigma \|_{L_T^2 L^2}^2\leq B_1,~\| v\|_{L_T^{\infty} H^2}^2 + \|v\|_{L_T^2 H^3}^2\leq B_2.\Big\}
		\end{align*}
where $B_1$ and $B_2$ are two constants defined in \eqref{eq_B1B2}. Without loss of generality, we assume that $T\leq 1$ in the sequel. Next, we show that for given $B_1$
and $B_2$ large enough,  $K_T\neq \varnothing$. Specifically, let $\tilde{\sigma}$ be the strong solution of the following problem 
\begin{equation*}
\begin{cases}
				\partial_t \tilde{\sigma}+\gamma\tilde{\sigma}-\varepsilon \Delta \tilde{\sigma} = 0, & \text{in } \Omega , \\
				\partial_n \tilde{\sigma} = 0, & \text{on } \partial \Omega  , \\
				\tilde{\sigma}(0) = \tau_0, & \text{in } \Omega.
\end{cases}
\end{equation*}
   Then, from Lemma \ref{lem_tau_para}, there exists a positive $C_1$, which is independent of $T$, such that 
   $$
    \|\tilde{\sigma}\|_{L_T^\infty H^1}^2 + \|\tilde{\sigma}\|_{L_T^2H^2}^2+ \|\partial_t \tilde{\sigma}\|_{L_T^2 L^2}^2  
 \leq C_1 \|\tau_0 \|_{H^1}^2.
   $$
Hence, we can choose 
\begin{equation}\label{bound_B_1_1}
    B_1\geq C_1 \|\tau_0 \|_{H^1}^2.
\end{equation}
Then, $(0,\tilde{\sigma}) \in K_T$ for any $0<T\leq 1$, which implies that $K_T\neq\varnothing$. Moreover, it is easy to check that, for any given $0<T\leq 1$, $K_T$  is a bounded, closed, and convex subset of  $X_T$, when endowed with the topology of $C([0,T];H^1)\times C([0,T];L^2)$. For given $(v,\sigma)\in K_T$, we define the mapping 
$$(u,\tau):=\Phi(v,\sigma),$$
where $(u,\tau)$ is the unique solution of the following linearized problem: 
\begin{equation}\label{eq_Phi}
 \begin{cases}
        -(1-\alpha) \mathbb{P}\Delta u =\mathbb{P}\div\sigma ,&\text{in } \Omega \times (0, T),\\ 
        \partial_t\tau+(v\cdot\nabla) \tau+\gamma \tau-\varepsilon\Delta\tau =\nabla v\sigma+\sigma\nabla^\top v+\alpha (\nabla v+\nabla^\top v),&\text{in } \Omega \times (0, T) ,\\
        \div u =0, &\text{in }\Omega\times(0,T),\\
					u =0,~~\partial_n \tau = 0, &\text{on }\partial\Omega\times(0,T), \\ 
     \tau = \tau_0,       &  \text{in }\Omega.
    \end{cases}
\end{equation}   
In order to apply the Schauder's fixed point theorem, we will prove that, for  appropriate choices of $B_1$ and $B_2$,   there exists a time $T^*> 0$  such that, for all $0<T \leq T^*$, $\Phi (K_T)\subset  K_T$.
\begin{lem}\label{lem_map}
		For appropriately chosen $B_1$ and $B_2$,  there exists a time $T^*> 0$  such that  $\Phi (K_T) \subset K_T$, for all $0<T \leq T^*$.
		\end{lem}
		\begin{proof}
			For $f=\mathrm{div}\sigma$ and $g=\sigma\nabla v+\nabla^\top v\sigma +\alpha(\nabla v+\nabla^\top v)$, we have 
            \begin{gather*}
                  \| f \|_{L_T^\infty L^2}^2 \leq C \| \sigma \|_{L_T^\infty H^1}^2\leq C B_1,~~\| f \|_{L_T^2H^1}^2 \leq C \| \sigma \|_{L_T^2H^2}^2\leq C B_1,  \\
                \| g \|_{L_T^2L^2}^2 \leq T^* C  ( \| v \|_{L_T^\infty H^1}^2 + \| v \|_{L_T^\infty H^2}^2 \| \sigma\|_{L_T^\infty H^1}^2 )\leq T^* C ( B_2 + B_1B_2 ).
            \end{gather*} 
               Then, choosing $B_2\geq 1$, using Lemmas \ref{stoke} and \ref{lem_tau_para}, there exists a constant $C_2$ such that 
            \begin{align}\label{u}
         	&\|u\|_{L_T^{\infty} H^2}^2 + \|u\|_{L_T^2 H^3}^2\leq C_2 B_{1}.
         \end{align}  
        and 
          \begin{align} \label{tau_est}
&\|\tau\|_{L_T^\infty H^1}^2 + \|  \tau\|_{L_T^2H^2}^2+ \|\partial_t \tau\|_{L_T^2 L^2}^2 \nonumber\\
&\leq C_2\left( \|\tau_0 \|_{H^1}^2 + T^*  ( B_2 + B_1B_2 )  \right)\exp\left( C_2 T^*  B_2^2   \right),
\end{align} 
  for all $0<T\leq T^*$. Based on \eqref{u} and \eqref{tau_est}, we can choose 
   \begin{equation}\label{eq_B1B2}
        B_1:=  (C_1 + 2eC_2) \|\tau_0 \|_{H^1}^2,~~B_2:= C_2 B_1 +1.
    \end{equation} 
    Then, setting 
        \begin{equation}\label{Tstar}
        T^*:=\min\left\{1,\frac{1}{ C_2 B_2^2},\frac{B_1}{2eC_2( B_2 + B_1B_2 )}\right\},
    \end{equation}
    we can finish the proof.
		\end{proof} 
Next, we show that $\Phi $ is continuous.
		\begin{lem}\label{continous}
        Let $T^*$ be defined as in \eqref{Tstar}. Then, for every $0<T\leq T^*$, the mapping $\Phi:K_T\to X_T$ is continuous. 
		\end{lem}
		\begin{proof}					
 Let $(u^{1},\tau^{1}),(u^{2},\tau^{2})$ in $K_{T^*}$, and$(u^{i},\tau^{i})=\Phi(v^{i},\sigma^{i}),i=1,2.$ Define $(\tilde{v},\tilde{\sigma}):= (v^1-v^2,\sigma^1-\sigma^2)$ and $(\tilde{u},\tilde{\tau}):= (u^1-u^2,\tau^1-\tau^2)$. Then, by the definition of $\Phi$, one can find that  $(\tilde{u},\tilde{\tau})$ satisfying 
		\begin{equation}\label{EQ-t23h}
				\begin{cases}
					(1-\alpha) A\tilde{u}=\mathbb{P}{\rm div}\,\tilde{\sigma}, \\
					\tilde{\tau}_{t}+({v}^{2}\cdot\nabla)\tilde{\tau}+{\gamma}\tilde{\tau}-\varepsilon\Delta\tilde{\tau}=-(\tilde{v}\cdot\nabla){\tau}^1+  \nabla \tilde{v}\sigma^1+\sigma^1\nabla^\top\tilde{v}\\
					\qquad\qquad\qquad\qquad\qquad\qquad~~+ \nabla {v}^2\tilde{\sigma} +\tilde{\sigma}\nabla^\top{v}^2+\alpha(\nabla\tilde{v}+\nabla^\top\tilde{v}).
				\end{cases}
			\end{equation}
			Taking the $L_2$-inner product of equation $\eqref{EQ-t23h}_2$ with $\tilde{\tau}$,  integrating by parts, we have		
			\begin{align} 
				&\frac{1}{2}\frac{\mathrm{d}}{\mathrm{d}t}||\tilde{\tau}||_{L^2}^2+\gamma||\tilde{\tau}||_{L^2}^2+\varepsilon||\nabla\tilde{\tau}||_{L^2}^2\nonumber\\
				\leq\,&\int  |(\tilde{v}\cdot\nabla){\tau}^1||\tilde{\tau}|+2 | \nabla\tilde{v}||{\sigma}^1||\tilde{\tau}| +2 ||\nabla{v}^2||\tilde{\sigma}||\tilde{\tau}| +2\alpha |\nabla\tilde{v}||\tilde{\tau}|\mathrm{d}x\mathrm{d} y \nonumber\\ 
                \leq\,& \frac{\gamma}{2} \|\tilde{\tau}\|^2_{L^2} + C\left( \|\nabla{\tau}^1\|^2_{L^4}\| \tilde{v}\|^2_{L^4}+ \|\sigma^1\|^2_{L^\infty} \|\nabla\tilde{v}\|^2_{L^2}  + \|\nabla v^2\|^2_{L^4} \|\tilde{\sigma}\|^2_{L^4} + \|\nabla\tilde{v}\|^2_{L^2}\right) \nonumber\\
                \leq\,& \frac{\gamma}{2} \|\tilde{\tau}\|^2_{L^2} + C\left(\| \tau^1\|^2_{H^2} +\| \sigma^1\|^2_{H^2}+\| v^2\|^2_{H^2}+1\right)\left( \| \tilde{v}\|^2_{H^1}+\| \tilde{\sigma}\|^2_{H^1}\right), \nonumber
			\end{align}
            which implies  that 
            			\begin{align} \label{eq_tauL2}
				& \frac{\mathrm{d}}{\mathrm{d}t}||\tilde{\tau}||_{L^2}^2+\gamma||\tilde{\tau}||_{L^2}^2+\varepsilon||\nabla\tilde{\tau}||_{L^2}^2\nonumber\\ 
                \leq\,&  C\left(\| \tau^1\|^2_{H^2} +\| \sigma^1\|^2_{H^2}+\| v^2\|^2_{H^2}+1\right)\left( \| \widetilde{v}\|^2_{H^1}+\| \widetilde{\sigma}\|^2_{H^1}\right). 
			\end{align} 
            Integrating the above inequality over $[0,T]$, we deduce that
			\begin{align}\label{tauH1}
				\|\widetilde{\tau}\|_{L_T^\infty L^2}^2&\leq C (\|{\tau}^1\|^2_{L_T^2H^2}+||\sigma^1||_{L_T^2H^2}^{2}+||v^2||_{L_T^2H^2}^{2}+T)(\|\tilde v\|_{L_T^\infty H^1}^{2}+\|\tilde\sigma\|_{L_T^\infty H^1}^{2}),\nonumber\\
				& \leq C(\alpha,\varepsilon,\gamma,B_1,B_2)   \left( \| \tilde{\sigma} \|_{L_T^{\infty}H^1}^2 + \| \tilde{v} \|_{L_T^\infty H^1}^2 \right),
			\end{align}
            where the fact $\tilde\tau_{0}= \tau_{0}^1- \tau_{0}^2=0$ is used.  For the estimate of $\tilde{u}$, using Lemma \ref{stoke}, we have
			\begin{equation}\label{u_H1}
				\|\tilde{u}\|_{L_T^{\infty}H^{1}}^{2}\leq C(\alpha)  \| \tilde{\sigma}  \|_{L_T^{\infty}L^2}^2 .
			\end{equation} 
            Using the estimate \eqref{tauH1} and \eqref{u_H1}, one can easily verify the continuity of $\Phi:K_T\to X_T$. The proof is complete.  
		\end{proof}
		\begin{lem}\label{compact}
        Let $T^*$ be defined as in \eqref{Tstar}. Then, for every $0<T\leq T^*$, the mapping $\Phi:K_T\to X_T$ is compact. 
		\end{lem}
        \begin{proof}
            Based on Lemma \ref{stoke} and the definition of $K_T$ and $X_T$, for $(v,\sigma)\in K_T$, one can verify that  $(u,\tau):=\Phi(v,\sigma)$ satisfies 
            \begin{gather*}
                u\in L^\infty(0,T;H^2),~\partial_tu\in L^2(0,T;L^2),\\
                \tau\in L^\infty(0,T;H^1),~\partial_t\tau\in L^2(0,T;L^2).
            \end{gather*}
            Here the bound for $\partial_tu$ follows by differentiating the Stokes equation in time and using $\partial_t\sigma\in L^2(0,T;L^2)$.
            Then, using Lemma \ref{Simom}, one can finish the proof directly.
        \end{proof}
		\begin{proof}[\bf Proof of Proposition \ref{Local_Existence}]
			Combining Lemmas~\ref{lem_map}, \ref{continous}, and \ref{compact} with Schauder's fixed point theorem yields the existence of solutions to problem \eqref{EQ-all}--\eqref{boundary-con} enjoying the regularity asserted in Proposition~\ref{Local_Existence}. Next, we prove the uniqueness of the solutions. 
            Let $(u ,\tau )$ and $(\bar{u},\bar{\tau})$ be two fixed points of the map $\Phi$. Set $W=u -\bar{u},$ and $Z=\tau -\bar{\tau}$, one can find that $(W,Z)$ satisfies the following problem: 
		\begin{equation}\label{th3}
			\begin{cases}
				-(1-\alpha)\Delta W=\div Z,  \\
				\partial_{t} Z+\bar{u} \cdot \nabla Z+\gamma Z-\varepsilon \Delta Z=-(W\cdot\nabla)\tau+\nabla W \tau+
				\tau\nabla^\top W+\nabla \bar{u}  Z \\
				\qquad\qquad\qquad\qquad\qquad\qquad~~~ +Z\nabla^\top \bar{u}+ \alpha(\nabla W+\nabla^\top W),  
                \end{cases}
		\end{equation}
        with initial and boundary conditions that 
        		\begin{align*} 
			 Z(x,y,0) = 0,~~(x,y)\in\Omega;~~W=0,~\partial_n Z=0, ~~(x,y,t)
\in\partial\Omega\times(0,T^*]. 		\end{align*}
		Taking the $L^2$-inner product of   $\eqref{th3}_1$ with $\alpha W$, using integration by parts, we have
		\begin{equation*} 
			\alpha(1-\alpha)\|\nabla W\|_{L^2}^2=-\alpha\int \nabla W:Z {\rm d}x{\rm d}y \leq \frac{\alpha(1-\alpha)}{2}\|\nabla W\|_{L^2}^2 + \frac{\alpha}{2(1-\alpha)}\|Z\|_{L^2}^2,
		\end{equation*} 
        which implies
        \begin{equation}\label{th4}
			\alpha(1-\alpha)\|\nabla W\|_{L^2}^2  \leq   \frac{\alpha}{(1-\alpha)}\|Z\|_{L^2}^2,
		\end{equation} 
		Taking the $L^2$-inner product of  $\eqref{th3}_2$ with $Z$ and add the result with $\frac{\gamma(1-\alpha)}{2\alpha}$\eqref{th4}, using integration by parts, we obtain
		\begin{align}\label{J}
			&\frac{1}{2}\frac{\mathrm{d}}{\mathrm{d}t}\|Z\|_{L^2}^2+\frac{\gamma}{2}\|Z\|_{L^2}^2+\varepsilon\|\nabla Z\|_{L^2}^2+ \frac{\gamma(1-\alpha)^2}{2}\|\nabla W\|_{L^2}^2\nonumber\\
           \leq & -\langle(W\cdot\nabla)\tau ,Z\rangle + \langle\nabla W \tau,Z\rangle+ \langle\tau\nabla^\top W,Z\rangle+ \langle\nabla \bar{u}  Z,Z\rangle\nonumber\\
           & + \langle Z\nabla^\top \bar{u}  ,Z\rangle+ \alpha\langle(\nabla W+\nabla^\top W),Z\rangle:=\sum_{i=1}^6J_i. 
		\end{align} 
		Using Sobolev inequality, H\"{o}lder's  inequality, Young's  inequality and Ladyzhenskaya inequality, we can estimate $J_i$ as follows
		\begin{align*}
        |J_1| &\lesssim  \|W\|_{L^4} \|\nabla\tau\|_{L^4} \|Z\|_{L^2} \leq \frac{\gamma(1-\alpha)^2}{12}\|\nabla W\|_{L^2}^2 + C\|\nabla\tau \|_{L^4}^2 \|Z\|_{L^2}^2,\\
			|J_2|+|J_3| &\lesssim \|\nabla W\|_{L^2} \|\tau \|_{L^\infty} \|Z\|_{L^2} \leq \frac{\gamma(1-\alpha)^2}{12}\|\nabla W\|_{L^2}^2 + C\|\tau\|_{H^2}^2 \|Z\|_{L^2}^2, \\
			|J_4|+|J_5|& \lesssim \|\nabla \bar{u}\|_{L^2} \|Z\|_{L^4}^2   \leq \frac{\varepsilon}{2} \|\nabla Z\|_{L^2}^2 +  C(\|\bar{u}\|_{H^1}^2 +1) \|Z\|_{L^2}^2.\\
			|J_6| &\lesssim \|\nabla W\|_{L^2}  \|Z\|_{L^2} \leq \frac{\gamma(1-\alpha)^2}{12}\|\nabla W\|_{L^2}^2+ C  \|Z\|_{L^2}^2. 
		\end{align*}
	 Substituting above estimates into \eqref{J}, we obtain
		\[\frac{\mathrm{d}}{\mathrm{d}t} \|Z\|_{L^2}^2 \lesssim \left(1+ \|\bar{u}\|_{H^1}^2+\| \tau \|_{H^2}^2\right)\|Z\|_{L^2}^2.\]
		Applying the Gronwall inequality and Stokes estimate, we find that $(W,Z)=0$. The proof is complete. 
		\end{proof}
		 
	\section{Uniform estimates}\label{Sec_est} 
		\begin{lem}\label{thm4}
			Under the assumptions of Theorem \ref{main_thm}, let $(u, \tau)$ be the solution of problem \eqref{EQ-all}-\eqref{boundary-con} on $[0, T]\times \Omega$. Then, for all $t\in[0, T]$, we have 
			$$
			\| u \|_{L_t^\infty H^1}^2+\| \tau \|_{L_t^\infty L^2}^2 + \|u\|_{L_t^2 H^2}^2+ \|\tau\|_{L_t^2 H^1}^2 \leq D_1, 
			$$
			where $D_1$ is defined by \eqref{D_1}.
		\end{lem}	
		\begin{proof}
			Multiplying $(\ref{EQ-t19h})_1$ by $2u$ and integrating the result over $\Omega$, we have
			\begin{align}\label{u-2u-eq}
				 2(1-\alpha)\|\nabla u\|_{L^2}^2=2\int\div\mathbb{T}\cdot u {\rm d}x{\rm d}y.
			\end{align}
			Taking  trace of $(\ref{EQ-t19h})_2$ and integrating the result over $\Omega$ by parts, we obtain
			\begin{equation}\label{4.0}
				\frac{\mathrm{d} }{\mathrm{d}t}||\mathrm{tr}\mathbb{T}||_{L^1}+\gamma||\mathrm{tr}\mathbb{T}||_{L^1}   =\int \mathrm{tr}(\nabla u\mathbb{T}+\mathbb{T}\nabla^\top u) {\rm d}x{\rm d}y + 2\alpha\gamma |\Omega|.
			\end{equation}
			Furthermore, from integration by parts, we observe that 
			$$\int\mathrm{tr}(\nabla u\mathbb{T}+\mathbb{T}\nabla^\top u) {\rm d}x{\rm d}y=-2\int \div\mathbb{T}\cdot u {\rm d}x{\rm d}y.$$
			Therefore, combining above result with \eqref{u-2u-eq}-\eqref{4.0}, we have
			\begin{equation}\label{EQ-t20h}
				\frac{\rm d}{\rm dt}\left\|\mathrm{tr}\mathbb{T}\right\|_{L^{1}}+\gamma||\mathrm{tr}\mathbb{T}||_{L^1}  +2 (1-\alpha)\|\nabla u\|_{L^{2}}^{2} = 2\alpha\gamma |\Omega|.
			\end{equation}
			Integrating \eqref{EQ-t20h} over $[0, t]$, using Corollary \ref{cor_L1}, we further obtain 
			\begin{equation}\label{EQ-t24h}
				\|\tau(t)\|_{L^{1}}+\gamma\|\tau(t)\|_{L^1_tL^{1}}+2(1-\alpha)\|\nabla u\|_{L_t^2L^2}^{2}\leq  C(T+ \|\tau_{0}\|_{L^{1}}). 
			\end{equation}
			Next, multiplying $(\ref{EQ-all})_2$ by $\tau$ over $\Omega$, we obtain
			\begin{align}\label{EQ-t26h}
				\frac{1}{2}\frac{\rm d}{\rm dt}\|\tau\|_{L^2}^2+\gamma\|\tau\|_{L^2}^2+\varepsilon\|\nabla\tau\|_{L^2}^2
				=\int(\nabla u\tau+\tau\nabla^\top u):\tau+\alpha(\nabla u+\nabla^\top u):\tau\,{\rm d}x{\rm d}y,			
			\end{align}
			where the two terms on the right-hand side can be estimated by applying H\"{o}lder's inequality, Young's inequality, and the Ladyzhenskaya inequality as follows
			\begin{align*}
				&\left|\int(\nabla u\tau+\tau\nabla^\top u):\tau\,{\rm d}x{\rm d}y\right|\leq2\|\nabla u\|_{L^2}\|\tau\|_{L^4}^2
				\leq\frac{\varepsilon}{2}\|\nabla\tau\|_{L^2}^2+C(\varepsilon)(\|\nabla u\|_{L^2}+\|\nabla u\|_{L^2}^2)\|\tau\|_{L^2}^2.\\
				&\left|\int\alpha(\nabla u+\nabla^\top u):\tau\,{\rm d}x{\rm d}y\right|\leq2\alpha\|\tau\|_{L^2}\|\nabla u\|_{L^2}
				\leq\frac{\gamma}{2}\|\tau\|_{L^2}^2+C(\alpha,\gamma)\|\nabla u\|_{L^2}^2.
			\end{align*}
			Substituting above estimates into \eqref{EQ-t26h}, we get 
			\begin{align*} 
				 \frac{\rm d}{\rm dt}\|\tau\|_{L^2}^2+\gamma\|\tau\|_{L^2}^2+\varepsilon\|\nabla\tau\|_{L^2}^2
				\lesssim (\|\nabla u\|_{L^2}+\|\nabla u\|_{L^2}^2)\|\tau\|_{L^2}^2 + \|\nabla u\|_{L^2}^2,			
			\end{align*}
            which, together with \eqref{EQ-t24h} and Gronwall inequality, implies that 
			\begin{align*}
				 &\|\tau(t)\|_{L^2}^2+\|\tau\|_{L_t^2H^1}^2 \\ 
				&\leq  C(\alpha,\gamma,\varepsilon)\left( \|\tau_0\|_{L^2}^2+\|\nabla u\|_{L_t^2L^2}^2 \right)\exp(\|\nabla u\|_{L_t^2L^2}^2+\|\nabla u\|_{L_t^1L^2}), \\
				&\leq C(\alpha,\gamma,\varepsilon)\left( \|\tau_0\|_{L^2}^2+T+\|\tau_{0}\|_{L^{1}}\right)\exp(CT+C\|\tau_{0}\|_{L^{1}}  ). 
			\end{align*}
            Moreover, combining the above estimate with Lemma \ref{stoke}, we  obtain that 
            \begin{align}\label{D_1}
                &\| u \|_{L_t^\infty H^1}^2+\| \tau \|_{L_t^\infty L^2}^2+\|u\|_{L_t^2 H^2}^2   +\|\tau\|_{L_t^2H^1}^2 \notag\\ 
                \lesssim& \left( \|\tau_0\|_{L^2}^2+T+\|\tau_{0}\|_{L^{1}}\right)\exp(CT+C\|\tau_{0}\|_{L^{1}}  )=:D_1. 
            \end{align} 
			The proof   is   complete.
		\end{proof} 
		\begin{lem}\label{thm5}
			Under the assumptions of Theorem \ref{main_thm}, let $(u, \tau)$ be the solution of problem \eqref{EQ-all}-\eqref{boundary-con} on $[0, T]\times \Omega$. Then, for all $t\in[0, T]$, we have 
			$$\|u\|_{L_t ^\infty H^2}^2+\|\tau\|_{L_t ^\infty H^1}^2 + \|u\|_{L_t ^2 H^3}^2 + \|\tau\|_{L_t ^2 H^2}^2\leq  D_2 .
			$$
			where $D_2$ is defined in \eqref{D2}.
		\end{lem}		
		\begin{proof}
			Multiplying $(\ref{EQ-all})_2$ by  $-\Delta\tau$, and integrating the result over $\Omega$ by parts, we have
			\begin{align}\label{nabla-tau-L2-est}
				&\frac{1}{2}\frac{\rm d}{\rm d t}||\nabla\tau||_{L^2}^2+\gamma||\nabla\tau\|_{L^2}^2+\varepsilon||\Delta\tau||_{L^2}^2 \nonumber\\
				&\,\,=\int-[\alpha(\nabla u+\nabla^\top u)]:\Delta\tau-(\nabla u\cdot\tau+(\tau\cdot\nabla^\top) u):\Delta\tau+(u\cdot\nabla)\tau:\Delta\tau  {\rm d}x{\rm d}y,\nonumber\\
				&\,\,:=M_1+M_2+M_3 .
			\end{align}
			According to H\"{o}lder's inequality, Young's inequality, Sobolev inequality and Lemma \ref{thm4}, we can  estimate   $M_j~(j=1,2,3)$ as follows:
			\begin{align*}
				|M_1| \leq\,&2\alpha\|\Delta\tau\|_{L^2}\|\nabla u\|_{L^2}\leq\frac{\varepsilon}{6}\|\Delta\tau\|_{L^2}^2
				+C  D_1 , \\
				|M_2| \leq\,&\|\Delta\tau\|_{L^2}\|\tau\|_{L^4}\|\nabla u\|_{L^4}
                \lesssim \|\Delta\tau\|_{L^2}(\|\tau\|_{L^2} +\|\nabla \tau\|_{L^2})\| u\|_{H^2}  \\
                \leq\,&\frac{\varepsilon}{6}\|\Delta\tau\|_{L^2}^2
				+ CD_1\|u\|_{H^2}^2  +C\|u\|_{H^2}^2\|\nabla\tau\|_{L^2}^2. \\		
				|M_3|\leq\,&\|u\|_{L^\infty}\|\nabla\tau\|_{L^2}\|\Delta\tau\|_{L^2}
				\leq\frac{\varepsilon}{6}\|\Delta\tau\|_{L^2}^2+C\|u\|_{H^2}^2 \|\nabla\tau\|_{L^2}^2.
			\end{align*}
			Substituting the estimates of $M_1$-$M_3$  into 
			\eqref{nabla-tau-L2-est}, we find that
			$$  \frac{\mathrm{d}}{\mathrm{d}t} \| \nabla \tau \|_{L^2}^2+||\nabla\tau\|_{L^2}^2 +   \| \Delta \tau \|_{L^2}^2 \leq  C  \|u\|_{H^2}^2 \|\nabla\tau\|_{L^2}^2+C D_1 \left(1+\|u\|_{H^2}^2\right),  
			$$
            which, together with
			  Gronwall  inequality, implies that 
			\begin{align}\label{eq_tau_H1}
				\|\nabla\tau(t)\|_{L^{2}}^{2}+\|\nabla \tau\|_{L^{2}H^{1}}^{2}  \leq  
				C (D_1T+\|\tau_0\|_{H^{1}}^{ 2 }+  D_1^2 ) \mathrm{e}^{ CD_1}.
			\end{align}
            Then, combining \eqref{eq_tau_H1} with Lemmas \ref{stoke} and \ref{thm4}, we  obtain that  
			\begin{align}\label{D2}
			    &\|u\|_{L_t ^\infty H^2}^2+\|\tau\|_{L_t ^\infty H^1}^2 + \|u\|_{L_t ^2 H^3}^2 + \|\tau\|_{L_t ^2 H^2}^2,\nonumber\\
                &\lesssim  (D_1T+\|\tau_0\|_{H^{1}}^{ 2 }+ D_1^2     )\mathrm{e}^{ CD_1 }=:D_2 .
			\end{align}
			Thus, the proof of Lemma \ref{thm5} is finished.
		\end{proof}
		We proceed by deriving estimates for the time derivatives $(\partial_t u, \partial_t \tau)$ in Lemma \ref{thm6}.
		\begin{lem}\label{thm6}
			Under the assumptions of Theorem \ref{main_thm}, let $(u, \tau)$ be the solution of problem \eqref{EQ-all}-\eqref{boundary-con} on $[0, T]\times \Omega$. Then, for all $t\in[0, T]$, we have 
			$$ \| \partial_t \tau \|_{L_t^2 L^2}^2 + \|\partial_t u\|_{L_t^2 H^1}^2 \leq D_3,
			$$
			where $D_3$ is defined in \eqref{D3}.
		\end{lem}					
		\begin{proof}
			Multiplying $(\ref{EQ-all})_2$ by $\partial_{t}\tau$, and integrating the result over $\Omega$ by parts, we obtain
			\begin{align}\label{partial-t-tau-est}
				\| \partial_t \tau \|_{L^2}^2 + \frac{1}{2} \frac{\mathrm{d}}{\mathrm{d}t} (\gamma \| \tau \|_{L^2}^2+\varepsilon \| \nabla \tau \|_{L^2}^2)  &=\int \alpha(\nabla u+\nabla^\top u):\partial_{t}\tau+(\nabla u\tau+\tau\nabla^\top u):\partial_{t}\tau\nonumber\\
				& \quad-(u \cdot \nabla) \tau : \partial_t \tau{\rm d}x{\rm d}y := N_1 + N_2 + N_3 . 
			\end{align}
			Following the same approach as for $M_1$-$M_3$, we estimate $N_1$-$N_3$ as follows:
			\begin{align*}
				& |N_1| \leq \frac{1}{6} \| \partial_t \tau \|_{L^2}^2 + C \| \nabla u \|_{L^2}^2\leq \frac{1}{6} \| \partial_t \tau \|_{L^2}^2 + CD_1 .\\
				& |N_2| \leq \frac{1}{6} \| \partial_t \tau \|_{L^2}^2 + C \| \nabla u \|_{L^4}^2 \| \tau \|_{L^4}^2\leq \frac{1}{6} \| \partial_t \tau \|_{L^2}^2 + CD_2^2.\\
				& |N_3| \leq \frac{1}{6} \| \partial_t \tau \|_{L^2}^2 + C \| u \|_{L^4}^2 \| \nabla \tau \|_{L^4}^2\leq\frac{1}{6} \| \partial_t \tau \|_{L^2}^2 + CD_1\| \nabla \tau \|_{H^1}^2.
			\end{align*}
			Substituting the above estimates into \eqref{partial-t-tau-est}, integrating the result over $[0,t]$ using Lemma \ref{stoke}, we obtain
			\begin{align}\label{D3}
				\|\partial_t u\|_{L_t^2 H^1}^2+\| \partial_t \tau \|_{L_t^2 L^2}^2 + \| \tau(t) \|_{H^1}^2 \lesssim \|\tau_{0}\|_{H^1}^2+ T\left(D_1+D_2^2 \right)+D_1D_2 := D_3.
			\end{align} 
			This concludes the proof of Lemma \ref{thm6}.
		\end{proof}
		
		\section{The proof of Theorem \ref{main_thm}}	
		Under the assumptions of Theorem \ref{main_thm}, let $(u, \tau)$ be the solution of problem \eqref{EQ-all}-\eqref{boundary-con} on $[0, T]\times \Omega$. With the help of Lemmas \ref{thm4}-\ref{thm6} one can easily find that $(u,\tau)$ satisfies the regularity in Theorem \ref{main_thm}. Moreover, 
        there exists a constant $C_T$ such that 
		$$\|u\|_{L_t ^\infty H^2}^2+\|\tau\|_{L_t ^\infty L^1} +\|\tau\|_{L_t^\infty H^1}^2 + \|u\|_{L_t^2 H^3}^2 + \|\tau\|_{L_t^2 H^2}^2 + \|\partial_tu\|_{L_t^2 H^1}^2 + \|\partial_t\tau\|_{L_t^2 L^2}^2\leq C_T.
			$$
            Then, combining the above estimate and Proposition \ref{Local_Existence}, one can complete the proof of Theorem   \ref{main_thm} by the standard
continuity method. Refer to \cite{Constantin_2012} and \cite{Kupferman_2007} for detailed discuss.
		\section*{Data availability statement}
		\noindent No new data were created or analysed in this study.
\section*{Conflict of interest} 
\noindent The authors declare that they have no conflict of interest.	
		\section*{Acknowledgement} 
        The work of Y. H. Wang was partially supported by the National Natural Science Foundation of China grant 12401274 and the Natural Science Foundation of Hunan Province grant 2024JJ6302.
		The work of S. H. Zhang was supported by the
		Postgraduate Scientific Research Innovation Project of Hunan Province grant
		CX20250720.
		
	\end{document}